\documentclass[reqno,11pt]{amsart}
\usepackage{amssymb, amsmath}
\usepackage{amsthm, amsfonts,mathrsfs}
\usepackage[T1]{fontenc}
\usepackage{times}
\usepackage{graphicx}
\usepackage{subfigure}
\usepackage{epsfig}
\usepackage{xcolor}
\usepackage[compress]{cite}
\usepackage[all]{xy}
\usepackage[pagewise]{lineno}\nolinenumbers
\newtheorem{theorem}{Theorem}[section]

\newtheorem{definition}{Definition}[section]
\newtheorem{lemma}{Lemma}[section]

\newtheorem{remark}{Remark}[section]

\numberwithin{equation}{section}

\newcommand{\p}{\partial}

\newcommand{\normmm}[1]{{\left\vert\kern-0.25ex\left\vert\kern-0.25ex\left\vert #1\right\vert\kern-0.25ex\right\vert\kern-0.25ex\right\vert}}
\allowdisplaybreaks[4]
\begin{document}
\raggedbottom
\title[Modulational stability in a local model for shallow water waves]
{Modulational stability of the periodic traveling wave in a local model for shallow water waves}%
\author[L. Fan]{Lili Fan}%
\address[Lili Fan]{School of Mathematics and Statistics,
Henan Normal University, Xinxiang 453007, China}
\email{fanlily89@126.com, fanlili@htu.edu.cn}
\author[X. Zhang]{Xin Zhang}
\address[Xin Zhang]{School of Mathematics and Statistics,
Henan Normal University, Xinxiang 453007, China}
\email{zhangxinn1201@163.com}
\author[H. Gao]{Hongjun Gao}%
\address[Hongjun Gao]{School of Mathematics,
Southeast University, Nanjing 211189, China}
\email{hjgao@seu.edu.cn, Corresponding author.}
\begin{abstract}
In this paper, we investigate the modulational stability of periodic traveling waves in a local model for shallow water waves, which is an extended version of the Hunter-Saxton equation. We construct a family of small-amplitude periodic traveling waves for this local model and provide a parameterization of these waves. Using Floquet-Bloch theory, perturbation theory, and spectral analysis, we then establish the modulational stability of these background periodic traveling wave solutions. Finally, we analyze the modulational instability of another extended Hunter-Saxton equation with cubic nonlinearities, following a similar approach.
\end{abstract}

\thispagestyle{empty}
\maketitle

\setcounter{tocdepth}{1}



\section{Introduction}
Water wave models have long been regarded as fundamental tools for investigating various nonlinear dispersive phenomena in fluid mechanics, oceanography, and other related physical systems. These models not only depict the dynamic characteristics of wave propagation but also capture the interaction mechanism between nonlinearity and dispersion--two key factors that are crucial to the formation, stability, and evolutionary process of waves. In these models, the behavior of periodic and solitary traveling wave solutions, as well as their responses to small perturbations, has attracted considerable attention from researchers.

In the present study, we focus on the modulational stability of periodic traveling waves in the following local shallow water wave model
\begin{equation}\label{1.1}
\left(c^{2}-2 u\right)u_{x x} - (u_x)^{2}+u-2c u_{t x}= 0,
\end{equation}
where $u = u(t,x)$ is the surface elevation and $c > 0$ is the wave speed. The specific form of this local model was proposed in \cite{Locke}, based on conformal transformations of Euler's equations, and serves as an extension of the Hunter-Saxton equation derived in \cite{Hunter} to describe the dynamics of direction fields. As discussed in \cite{Natali1}, the local model \eqref{1.1} is closely related to the nonlocal Babenko equation \cite{Babenko} and the Camassa-Holm equation \cite{Alber,Alber1}.

In recent decades, research on local nonlinear wave equations has made remarkable achievements. For numerous classical local wave models, integrability has been established, and a series of effective analytical and numerical methods have been developed to explore their dynamical behaviors. Regarding the local model \eqref{1.1}, its integrability was established in \cite{Hone}, along with that of several other peaked wave equations, including the reduced Ostrovsky and short-pulse equations. Some traveling wave solutions of these peaked wave equations were explored in \cite{Matsuno} via Hirota's bilinear method and in \cite{Locke} using dynamical system techniques. Moreover, the local well-posedness of sufficiently smooth solutions in Sobolev spaces has been proven in \cite{Ye}.

Reflecting physical observations and predictions constitutes an essential aspect of mathematical research. Consequently, the investigation of the stability and instability of traveling waves has long been a focal point in both physics and mathematics. Here we consider the periodic traveling waves, which are spatially periodic solutions that propagate at a constant speed while maintaining a fixed profile. As an important class of traveling waves, periodic waves exhibit richer structural complexity due to their dependence on the period, wave speed, and integration constant. Owing to the inherent features of periodic waves, the investigation of their stability typically presents additional technical difficulties in comparison with the analysis of solitary waves. In particular, the periodicity of these waves gives rise to a broader class of perturbations, including co-periodic, multiple-periodic, and localized perturbations, among others.

For the local model \eqref{1.1}, it has been proved that traveling periodic waves with smooth, peaked, and cusped profiles exist, and that the families of smooth and cusped profiles are connected at the limiting wave with the peaked profile, as shown in \cite{Locke}. In that work, the authors also proved the linear stability of smooth periodic waves with respect to co-periodic perturbations. A subsequent study considered the linear and nonlinear instability of the limiting periodic wave with the peaked profile in $H^1_{\textrm{per}}\cap W^{1, \infty}$ \cite{Natali1}. The perturbations considered in the present paper are localized perturbations; that is, we are concerned with the modulational stability of periodic traveling waves within the local model \eqref{1.1}.

Modulational instability, also referred to as Benjamin-Feir instability, is a classical problem in the study of periodic traveling waves, concerning whether these waves remain stable or become unstable under long-wavelength perturbations. This problem can be traced back to the 1960s, when Benjamin and Feir \cite{Benjamin, Benjamin1} discovered that a small amplitude space periodic Stokes wave is unstable under a long-wavelength perturbation. The first rigorous proof of the modulational instability for the Stokes waves was given by Bridges-Mielke \cite{Bridges} in 1995 for the finite-depth case, and later by Nguyen-Strauss \cite{Nguyen} in 2020 for the infinite-depth case. Recently, Berti et al. \cite{Berti} established a complete characterization of the figure-8 pattern of unstable eigenvalues in the deep-water regime, followed by corresponding results in the finite-depth case \cite{Berti1} and the delicate critical-depth case \cite{Berti2}. In addition to the Stokes waves mentioned above, modulational instability has also been observed in several important nonlinear dispersive models, including the studies on modulational instability of the Whitham equation \cite{Hu}, of the generalized KdV equations \cite{Maspero}, of the general Benjamin-Bona-Mahony (BBM) and the regularized Boussinesq types equations \cite{HP,H1}, of the full-dispersion Camassa-Holm equations \cite{HP1}, of the Novikov equation \cite{Ehrman1}, of the $b$-family equation\cite{Fan}, and of dispersive PDE models \cite{Jin}.

It should be noted that although some of the aforementioned models involve nonlocal operators--such as the BBM equation and the Camassa-Holm equation-these nonlocal operators are well defined in the space of square integrable functions. The same does not hold for the local model \eqref{1.1}, which can be rewritten in the form
\begin{equation}\label{1.4}
\left(2cu_t -c^2 u_x +2uu_x\right)_x = u + (u_x)^2.
\end{equation}
Clearly, $\partial_x$ is not invertible in the space of square integrable functions. It is then natural to consider a related model, namely the Ostrovsky equation
\begin{equation}\label{1.3}
\left(v_t + \beta v_{xxx} + (v^2)_x\right)_x = \alpha v, \quad \alpha >0,\quad\beta \in \mathbb{R} \setminus \{0\},
\end{equation}
where the variable has been renamed to avoid confusion with the solution $u$ of \eqref{1.4}. Taking $\beta =0$, \eqref{1.3} reduces to the reduced Ostrovsky equation, whose smooth, peaked, and cusped traveling periodic waves have also been deeply investigated. Notable results include the linear and nonlinear stability of smooth traveling periodic waves \cite{Geyer1, Hakkaev, Johnson16}, and the linear instability of the peaked traveling periodic waves \cite{Geyer2, Geyer3}.

Modulational instability related to the Ostrovsky equation has been investigated in \cite{Bhavna}, where the authors considered the dispersion-generalized Ostrovsky equation. Inspired by this recent work, we examine the modulational stability of periodic traveling waves within the local model \eqref{1.1}. It should be noted that the basic idea of our approach stems from \cite{Bhavna}, and our results reveal that the presence of the term $(u_x)^2$ on the right-hand side of \eqref{1.4} does not affect the modulational stability (see Remark \ref{re2}). On the other hand, our findings are consistent with the results reported in \cite{Locke}, where a $2\pi$-periodic wave of the local model \eqref{1.1} is spectrally stable to co-periodic perturbations (see Remark \ref{re1}).

Finally, the aforementioned approach is applied to analyze the modulational instability of another local model-an integrable equation governing short-waves in a long-wave model, as proposed in \cite{Fa}.
\begin{equation}\label{6.1}
u_{t x}+u u_{x x}+\frac{1}{2}(u_x)^{2}-\frac{\gamma}{2}u_{x x}(u_x)^{2}-u= 0,
\end{equation}
where $u(t, x)$ stands for the fluid velocity at the surface of water waves and $\gamma \in\mathbb{R}$ is a fixed physical parameter related to the dimensionless Bond number. For $\gamma=0$, the equation \eqref{6.1} reduces to the Hunter-Saxton equation. Moreover, \eqref{6.1} admits the following equivalent form
\begin{equation}\label{6.10}
\left(u_t+u u_x-\frac{\gamma}{6}(u_x)^3\right)_x=u+\frac{1}{2}(u_x)^2.
\end{equation}
The main difference between \eqref{6.10} and the local model \eqref{1.4} (i.e., \eqref{1.1}) lies in the replacement of the term $c^2 u_x$ by $\frac{\gamma}{6}(u_x)^3$ and the conclusions of Theorem \ref{the6.1}, together with those of Theorem \ref{the5.1}, reveal that the cubic nonlinear term $\frac{\gamma}{6}(u_x)^3$ induces modulational instability.

The remainder of this paper is organized as follows. In Section \ref{sec expression}, we establish the existence of a one-parameter family of small-amplitude periodic traveling waves to the local model \eqref{1.1} by employing the implicit function theorem and Lyapunov-Schmidt reduction, and give a parameterization of these waves. In Section \ref{sec Bloch}, we formulate the spectral stability problem based on the linearization of \eqref{1.1} about the obtained periodic traveling wave, and present the results on the Bloch-wave decomposition. In Section \ref{sec unper}, we investigate the nontrivial kernel of the unperturbed operators. In Section \ref{sec spectra}, by means of perturbation arguments, we obtain the modulational stability of the small-amplitude periodic traveling wave solutions to equation \eqref{1.1}. In Section \ref{Another model}, we present an analysis of the modulational instability for the local model \eqref{6.1}.
\subsection*{Notations}
Throughout this paper, we will use the following notations. The space $L^2(\mathbb{R})$ consists of all Lebesgue measurable, real or complex-valued functions $f : \mathbb{R}\rightarrow\mathbb{C}$ such that
\[
\|f\|_{L^2(\mathbb{R})}=\left(\int_{\mathbb{R}}|f(x)|^2 \mathrm{~d} x\right)^{1 / 2}<+\infty.
\]
Correspondingly, $L^2(\mathbb{T})$ denotes the space of $2 \pi$-periodic, measurable, real or complex-valued functions over $\mathbb{R}$ such that
\[
\| f \|_{L^2(\mathbb{T})}=\left(\frac{1}{2 \pi} \int_{0}^{2\pi}|f(x)|^2 \mathrm{~d} x\right)^{1 / 2}<+\infty,
\]
and $\|f\|_{L^{\infty}(\mathbb{T})}:=\operatorname{ess} \sup _{0\leq z <2\pi}|f(z)|<\infty$, where $\operatorname{ess} \sup$ represents the essential supremum over the periodic domain. The first-order periodic Sobolev space $H^1(\mathbb{T})$ is defined as
$$
H^1(\mathbb{T})= \left\{f\in L^2(\mathbb{T})| \p_z f\in L^2(\mathbb{T})\right\}.
$$
For any integer $k\geq0$, the $k$-th order Sobolev space $H^k(\mathbb{T})$ can be defined recursively. We further denote the infinite-order smooth periodic function space by the intersection $H^{\infty}(\mathbb{T})=\bigcap_{k=0}^{\infty} H^k(\mathbb{T})$.
For $f \in L^1(\mathbb{T})$, the Fourier series of $f$ is defined by
$$
\sum_{n \in \mathbb{Z}} \widehat{f}_n \mathrm{e}^{\mathrm{i} n z}, \quad \text { where } \widehat{f}_n=\frac{1}{2 \pi} \int_{0}^{2\pi} f(z) \mathrm{e}^{-\mathrm{i} n z} \mathrm{~d} z .
$$
If $f \in L^2(\mathbb{T})$, then its Fourier series converges to $f$ pointwise almost everywhere. We define the $L^2(\mathbb{T})$-inner product as
\begin{equation}\label{1.2}
\langle f, g\rangle=\frac{1}{2 \pi} \int_{0}^{2\pi} f(z) \bar{g}(z) \mathrm{d} z=\sum_{n \in \mathbb{Z}} \widehat{f}_n \overline{\widehat{g}}_n.
\end{equation}
Moreover, the commutator of two operators acting on a Hilbert space is defined by the following
$$
[F, G]:=F \circ G-G \circ F.
$$
\section{Asymptotically small-amplitude periodic traveling waves}\label{sec expression}
We begin by seeking periodic traveling wave solutions to \eqref{1.1}. Traveling waves of the equation \eqref{1.1} are solutions of the form
\[
u(x, t) = u(x - ct),
\]
where $c>0$ is the speed of propagation, and $u$ satisfies the ODE
$$
\left(3c^{2}-2 u\right)u^{\prime \prime} - (u^{\prime})^{2}+u= 0.
$$
Let $u$ be a $2 \pi / k$-periodic traveling wave solution of \eqref{1.1}, for some wave number $k > 0$. Then, $\eta(z):=u(k(x-ct))$ is a $2 \pi$-periodic function in $z$, satisfying
\begin{equation}\label{2.1}
\left(3 c^{2}-2 \eta\right) k^{2} \eta^{\prime \prime} - k^{2}(\eta^{\prime})^{2}+\eta = 0.
\end{equation}
Note that \eqref{2.1} is invariant under translations $(z \mapsto z+z_0$, $z \mapsto -z)$ for any $z_0\in\mathbb{R}$ and therefore, we may assume that $\eta$ is even.

Let $F:H^{2}_{\text{even}}(\mathbb{T}) \times\mathbb{R}^{+}\times\mathbb{R}^{+}\rightarrow L^2(\mathbb{T})$ be defined as
\begin{equation}\label{2.2}
F(\eta, c; k)=\left(3 c^{2}-2 \eta\right) k^{2} \eta^{\prime \prime} - k^{2}(\eta^{\prime})^{2}+\eta.
\end{equation}
We seek a non-trivial $2 \pi$-periodic solution $\eta\in H^{2}_{\text{even}}(\mathbb{T})$ of \eqref{1.1} such that
\begin{equation}\label{2.3}
F(\eta, c; k)=0.
\end{equation}
Noting that $\eta\in H^{2}_{\text{even}}(\mathbb{T})$ solves \eqref{2.1}, by the Sobolev inequality and a bootstrap argument, it follows that solutions $\eta\in H^{2}_{\text{even}}(\mathbb{T})$ of \eqref{2.1} necessarily satisfy $\eta\in H^{\infty}_{\text{even}}(\mathbb{T})$.

To study \eqref{2.1}, note that $F(0, c; k)=0$ for all $c>0$ and $k>0$ and
\begin{equation*}
\partial_\eta F\left(0, c; k\right)=3 c^{2} k^{2}\partial^{2}_{z}+1,
\end{equation*}
so that, in particular, we have
\begin{equation*}
\partial_\eta F\left(0, c; k\right) \cos(nz)=\left(1-3 c^{2} k^{2}n^{2}\right)\cos(nz).
\end{equation*}
It follows that $\ker\big(\partial_\eta F\left(0, c_0; k\right)\big)=\operatorname{span}(\cos(z))$, provided that
\begin{equation}\label{2.3a}
c=c_0=\frac{1}{\sqrt{3} k}.
\end{equation}
Using a Lyapunov-Schmidt argument, one can thus establish the existence of a one-parameter family of non-trivial, even solutions $\big(\eta(a;k)(\cdot), c(a;k)\big)$ of \eqref{2.1} for $|a|\ll 1$. This existence argument is elementary and follows the same lines as those in \cite{HP,HP1}, and is hence omitted here. The small-amplitude expansion of these solutions is given as follows and the details are provided in Appendix \ref{app A}.
\begin{lemma}\label{le1}
For all wavenumbers $k>0$, there exists a one parameter family of solutions of \eqref{2.1} given by $u(x, t)=\eta(a; k)\big(k(x-c(a; k)t)\big)$ for $a\in \mathbb{R}$ and $|a|$ sufficiently small; $\eta(a; k)(\cdot)$ is $2 \pi$-periodic, even and smooth in its argument, and $c(a; k)$ is even in $a$; $\eta(a; k)(\cdot)$ and $c(a; k)$ depend analytically on $a$ and $k$. Moreover, as $a \to 0$,
\begin{equation}\label{2.4}
\eta(a; k)(z)= a \cos\left(z\right)+a^{2}\left(A_0+A_2\cos\left(2z\right)\right)+a^{3} A_3\cos\left(3z\right)+O\left(a^4\right),
\end{equation}
\begin{equation}\label{2.5}
c(a; k)=c_0+a^2 c_2+O\left(a^4\right),
\end{equation}
where
\begin{equation}\label{2.6}
\begin{array}{ll}
\displaystyle A_0=-\frac{k^2}{2},\quad A_2=-A_0,\quad A_3=\frac{7k^4}{16}, \\
\displaystyle c_0=\frac{1}{\sqrt3 k}, \quad c_2=\frac{k^3}{4\sqrt3}.
\end{array}
\end{equation}
\end{lemma}
\section{Bloch-wave decomposition}\label{sec Bloch}
Throughout this section, let $\eta = \eta(z; a, k)$ with $k>0$ and $|a| \ll 1$ be a small-amplitude $2 \pi$-periodic traveling wave solution of \eqref{2.1} with wave speed $c=c(a; k)$, whose existence follows from Lemma \ref{le1}. Linearizing equation \eqref{1.1} about $\eta$ given in \eqref{2.4}, and considering perturbations to $\eta$ in the form $\eta+\varepsilon v(t, z)$, we arrive at the linear evolution equation
\begin{equation}\label{3.1a}
-2c kv_{t z}-2k^{2}(\eta v)_{z z}+2k^{2}\eta_z v_z+3c^2 k^2 v_{z z}+v=0.
\end{equation}
For $v(z, t)=\mathrm{e}^{\frac{\lambda}{k}t} \widetilde{v}(z)$, $\lambda\in\mathbb{C}$ and $\widetilde{v}(z) \in L^2(\mathbb{R})$, we arrive at the equation
\begin{equation}\label{3.1}
\mathcal{T}^{\lambda}_{k, a} \widetilde{v}:= \left[2(c \lambda-k^2 \eta_z)\p_z+k^2 \p_{z z}\left(-3c^2+2\eta\right)-1\right]\widetilde{v}=0,
\end{equation}
where $\mathcal{T}^{\lambda}_{k, a}:H^2(\mathbb{R})\rightarrow L^2(\mathbb{R})$ is considered as a closed, densely defined linear operator.
\begin{definition}[Spectral stability]
For a $2 \pi/k$-periodic traveling wave solution $u(x, t)=\eta(k(x-c t))$ of \eqref{1.1}
where $\eta$ and $c$ are given in \eqref{2.4}-\eqref{2.5}, we say that the periodic wave $\eta$ is spectrally stable with respect to square integrable perturbations if the operator $\mathcal{T}^{\lambda}_{k, a}$ is invertible on $L^2(\mathbb{R})$ for every $\lambda\in \mathbb{C}$ with $\operatorname{Re}(\lambda)>0$. Otherwise, $\eta$ is deemed to be spectrally unstable.
\end{definition}
\begin{remark}\label{rem3.1}
Due to the fact that $\eta$ is an even function, and \eqref{3.1} is invariant under the transformations $(v, \lambda)\mapsto(\overline{v},\overline{\lambda})$ and $(z, \lambda)\mapsto(-z, -\lambda)$, the set of $\lambda \in \mathbb{C}$, for which the operator $\mathcal{T}^{\lambda}_{k, a}$ fails to be invertible, is symmetric with respect to both the real and imaginary axes. Then $\eta$ is spectrally stable if and only if $\mathcal{T}^{\lambda}_{k, a}$ is invertible for all $\lambda \in \mathbb{C}$ with $\operatorname{Re}(\lambda)\neq 0$.
\end{remark}
Since the coefficients of the operator $\mathcal{T}^{\lambda}_{k, a}$ are periodic functions, we invoke Floquet theory to characterize the solutions of \eqref{3.1} in $L^2(\mathbb{R})$. Specifically, every solution $\widetilde{v}(z) \in L^2(\mathbb{R})$ of \eqref{3.1} can be expressed in the form $\widetilde{v}(z) = \mathrm{e}^{i\mu z} V(z)$, where $\mu \in\left(-\frac{1}{2}, \frac{1}{2}\right]$ is the Floquet exponent and $V$ is a $2 \pi$-periodic function. For a similar situation, we refer the reader to \cite{H1}. This reduces the invertibility problem of the operator $\mathcal{T}^{\lambda}_{k, a}$ on $L^2(\mathbb{R})$ to a one-parameter family of invertibility problem of Bloch operators on $L^2(\mathbb{T})$, which is parameterized by the Floquet exponent $\mu$. A precise reformulation of this result is provided in the following lemma.
\begin{lemma}\label{lem3.1}
The linear operator $\mathcal{T}^{\lambda}_{k, a}$ is invertible with bounded inverse on $L^2(\mathbb{R})$ if and only if the Bloch operators
\begin{equation}\label{3.2}
\mathcal{T}^{\lambda}_{k, a, \mu}:=2\left(c \lambda-k^2 \eta_z\right)\left(\partial_z+i \mu\right)+k^2\left(\partial_z+i \mu\right)^2\left(-3c^2+2\eta\right)-1,
\end{equation}
acting in $L^2(\mathbb{T})$ with domain $H^2(\mathbb{T})$ are invertible for all $\mu \in(-\frac{1}{2}, \frac{1}{2}]$. Moreover, $\mathcal{T}^{\lambda}_{k, a, \mu}$ is invertible in $L^2(\mathbb{T})$ if and only if zero is not an $L^2(\mathbb{T})$-eigenvalue of $\mathcal{T}^{\lambda}_{k, a, \mu}$.
\end{lemma}
\begin{definition}[Modulational stability]
A periodic traveling wave solution $\eta(a; k)$ of \eqref{1.1} is said to be modulationally stable if the associated linear operators $\mathcal{T}^{\lambda}_{k, a, \mu}$ are invertible on $L^2(\mathbb{T})$ for all $|(\lambda,\mu)| \ll 1$
with $\operatorname{Re}(\lambda)\neq 0$. Otherwise, the solution $\eta(a; k)$ is modulationally unstable.
\end{definition}
It is straightforward to obtain the following symmetry property of the Bloch operators $\mathcal{T}^{\lambda}_{k, a, \mu}$.
\begin{lemma}[Symmetry property]
Assume that $\mu \in\left(-\frac{1}{2}, \frac{1}{2}\right]$. Then the Bloch operators $\mathcal{T}^{\lambda}_{k, a, \mu}$ acting on $L^2(\mathbb{T})$ satisfy
\begin{equation}\label{3.3}
\mathcal{T}^{\lambda}_{k, a, \mu}(z)=\overline{\mathcal{T}^{-\overline{\lambda}}_{k, a, \mu}(-z)}=\overline{\mathcal{T}^{\overline{\lambda}}_{k, a, -\mu}(z)}.
\end{equation}
\end{lemma}
We now initiate the study of the $L^2(\mathbb{T})$-kernel of the operator $\mathcal{T}^{\lambda}_{k, a, \mu}$ for $\mu \in\left(-\frac{1}{2}, \frac{1}{2}\right]$ and $|a|$ sufficiently small. From the above lemma, it suffices to consider $\mu \in\left[0, \frac{1}{2}\right]$. Since $k$ is fixed, in the subsequent discussion, we denote $\mathcal{T}^{\lambda}_{k, a, \mu}$ by $\mathcal{T}^{\lambda}_{a, \mu}$.

\section{Unperturbed operators}\label{sec unper}
We first consider the case $a=0$, which corresponds to the trivial solution $\eta=0$. A direct computation using Fourier series yields that
\begin{equation}\label{4.1}
\mathcal{T}^{\lambda}_{0, \mu}\mathrm{e}^{i n z}=\left(\frac{2\lambda}{\sqrt3 k}i(n+\mu)+(n+\mu)^2-1\right)\mathrm{e}^{i n z}=0,
\end{equation}
for all $n\in \mathbb{Z}$ and $\mu \in \left[0, \frac{1}{2}\right]$. The kernel of $\mathcal{T}^{\lambda}_{0, \mu}$ is thus non-trivial when
\begin{equation}\label{4.2}
\lambda=i\frac{\sqrt3 k}{2}\left(n+\mu-\frac{1}{n+\mu}\right)=:i\Omega_{n, 0, \mu},\quad n\in \mathbb{Z},
\end{equation}
and hence the trivial solution $\eta=0$ of \eqref{1.1} is spectrally stable with respect to square integrable perturbations, as expected. Moreover, for sufficiently small $|a|$, the values $\lambda$ given in \eqref{4.2} will bifurcate to leave imaginary axis only when there is a collision on imaginary axis, thereby inducing instability. Therefore, it is crucial to determine the precise locations of these values $i \Omega_{n, 0, \mu}$, particularly their multiplicities. 
Now we summarize in the following lemma the possible collisions of the values $i\Omega_{n, 0, \mu}$.
\begin{lemma}\label{lem4.1}
For $\mu\in[0,\frac{1}{2}]$, the values $i\Omega_{n, 0, \mu}$ have the following properties:
\begin{enumerate}
\item For $\mu=0$, the collision occurs only at $\Omega_{-1, 0, 0}=\Omega_{1, 0, 0}=0;$
\item\label{4.1.2} For $\mu\in(0,\frac{1}{2}]$, there exist
    collisions of values $i\Omega_{0, 0, \frac{-n-\sqrt{n^2-4}}{2}}$ and $i\Omega_{n, 0, \frac{-n-\sqrt{n^2-4}}{2}}$ for $n\leq-3$.
\end{enumerate}
\end{lemma}
\begin{proof} Assume that there exists some $\mu_0 \in[0,\frac{1}{2}]$ such that $\Omega_{n, 0, \mu_0}=\Omega_{m, 0, \mu_0}$ for $m, n\in \mathbb{Z}$. From \eqref{4.2}, we get that
\begin{equation}\label{4.2a}
\left(n+\mu_0\right)\left(m+\mu_0\right)=-1.
\end{equation}
For $\mu_0=0$, it follows from \eqref{4.2a} that $n m=-1$ and thus the collision occurs at $\Omega_{-1, 0, 0}=\Omega_{1, 0, 0}=0$.

The monotonicity of the function $f(x)=x-\frac{1}{x}$ yields that
$$\cdots<\Omega_{-3, 0, \mu}<\Omega_{-2, 0, \mu}<0<\Omega_{1, 0, \mu}<\Omega_{2, 0, \mu}<\Omega_{3, 0, \mu}<\cdots.$$
Now we consider $\mu_0 \in(0,\frac{1}{2}]$. As $\Omega_{-1, 0, \mu_0}=-1+\mu_0-\frac{1}{-1+\mu_0}>0$ and $\Omega_{0, 0, \mu_0}=\mu_0-\frac{1}{\mu_0}<0$, there are two types of possible collisions: of $i\Omega_{0, 0, \mu_0}$ and $i\Omega_{n, 0, \mu_0}, n\leq -2$, and of $i\Omega_{-1, 0, \mu_0}$ and $i\Omega_{n, 0, \mu_0}, n\geq 1$.
\begin{enumerate}
\item\label{m=0}{For the collisions of $i\Omega_{0, 0, \mu_0}$ and $i\Omega_{n, 0, \mu_0}, n\leq -2$, we get from \eqref{4.2a} that $n+\mu_0=-\frac{1}{\mu_0}\leq -2$, which implies $n<-2$. Moreover, \eqref{4.2a} reduces to $\mu_0^2+n \mu_0+1=0$, so that
\begin{equation}\label{4.2b}
0< \mu_0=\frac{-n\pm\sqrt{n^2-4}}{2}\leq\frac{1}{2}.
\end{equation}
\begin{enumerate}
\item For $0< \mu_0=\frac{-n+\sqrt{n^2-4}}{2}\leq\frac{1}{2}$, which implies $0<\sqrt{n^2-4}\leq n+1$. Hence $-1<n<-2$, which is a contradiction;
\item For $0< \mu_0=\frac{-n-\sqrt{n^2-4}}{2}\leq\frac{1}{2}$, which implies $ -n-1\leq\sqrt{n^2-4}$. Then we get $n\leq-3$, which yields the property \eqref{4.1.2}.
\end{enumerate}
}
\item\label{m=1}{For the collisions of $i\Omega_{-1, 0, \mu_0}$ and $i\Omega_{n, 0, \mu_0}, n\geq 1$, we find that $n+\mu_0=\frac{-1}{-1+\mu_0}\in (1,2]$. Hence, $n$ can only equal $1$, and substituting this into \eqref{4.2a} yields $\mu_0=0$. This leads to a contradiction, thereby completing the proof.}
\end{enumerate}
\end{proof}
The above lemma shows that any of these collisions may lead to unstable solutions. Here, we are concerned analytically with the relatively manageable case of modulational stability/instability, i.e., we will consider small values of $\mu>0$ for $i\Omega_{n, a, \mu}$ in the vicinity of the origin in $\mathbb{C}$.
\section{Modulational Stability}\label{sec spectra}
In this section, we investigate the modulational stability of the wave $\eta$. First,
we aim to track how the values $\Omega_{\pm 1, a, 0}$ bifurcate from the origin for $|(a, \mu)|\ll1$. Note that the operator $\mathcal{T}^{\lambda}_{a, \mu}$ is a perturbation of $\mathcal{T}^{\lambda}_{0, 0}$, and that
$$
\left\|\mathcal{T}^{\lambda}_{a, \mu}-\mathcal{T}^{\lambda}_{0, 0}\right\|_{H^2(\mathbb{T})\rightarrow L^2(\mathbb{T})}=O(|a|+|\mu|),
$$
holds uniformly in operator norm as $a, \mu\rightarrow 0$. From Lemma \ref{lem4.1}, the two-dimensional generalized kernel for $\mathcal{T}^{0}_{0, 0}$ can be continued into a two-dimensional critical subspace
\begin{equation*}
\Sigma_{0, \mu} = \ker\left(\mathcal{T}^{i \Omega_{1, 0, \mu}}_{0, \mu}\right)\oplus \ker\left(\mathcal{T}^{i \Omega_{-1, 0, \mu}}_{0, \mu}\right),
\end{equation*}
which admits a $\mu$-independent orthogonal basis
\begin{equation}\label{5.1}
\varphi_1(z)=\cos(z), \quad \varphi_2(z)=\sin(z).
\end{equation}
Accordingly, for $|(a, \mu)|\ll1$ there can be only two values $\lambda=i \Omega_{\pm 1, a, \mu}$ in a sufficiently small neighborhood of the origin such that the operator $\mathcal{T}^{\lambda}_{a, \mu}$ is non-invertible.

According to Kato's analytic perturbation theory for operators, the functions
$(a, \mu)\mapsto\Omega_{\pm 1, a, \mu}$, are analytic in $(a, \mu)$ for $|(a, \mu)|\ll1$ and limit back to $\Omega_{\pm 1, 0, 0}=0$ as $(a, \mu)\rightarrow(0, 0)$. Furthermore, there exists a two-dimensional critical subspace
\begin{equation*}
\Sigma_{a, \mu} = \ker\left(\mathcal{T}^{i \Omega_{1, a, \mu}}_{a, \mu}\right)\oplus \ker\left(\mathcal{T}^{i \Omega_{-1, a, \mu}}_{a, \mu}\right),
\end{equation*}
which is an analytical continuation of that found at $a=0$ above. Now, we give the main result of the present paper.
\begin{theorem}\label{the5.1}
A sufficiently small $2\pi/k$-periodic traveling wave of \eqref{1.1} is modulationally stable.
\end{theorem}
\begin{proof}
To track the critical values $\lambda(a, \mu)=i\Omega_{\pm1, a, \mu}$ for $|(a, \mu)|\ll1$, we project the operator equation $\mathcal{T}^{\lambda}_{a, \mu} v=0$ onto the two-dimensional critical subspace $\Sigma_{a, \mu}$ defined above. More precisely, we will construct a suitable basis $\left\{\varphi_j(z; a, \mu)\right\}_{j = 1, 2}$ for $\Sigma_{a, \mu}$ and compute the $2 \times 2$ corresponding matrix
\begin{equation}\label{5.2}
\mathbf{B}_{a, \mu}^{\lambda} =\left(\frac{\left\langle\mathcal{T}^{\lambda}_{a, \mu} \varphi_i, \varphi_j\right\rangle}{\left\langle\varphi_i, \varphi_i\right\rangle}\right)_{i, j=1,2}.
\end{equation}
We now calculate the expansions of the basis $\{\varphi_1, \varphi_2\}$ for $\Sigma_{a, \mu}$ in $L^2(\mathbb{T})$. For $\mu=0$ and small $a$, we use the expansions of $\eta$ and $c$ given in \eqref{2.4}-\eqref{2.5}, and have
\begin{align}
\varphi_1(z; a, 0)&=-\frac{1}{a}\left(\partial_z \eta\right)(z)=\sin z+a k^2 \sin 2z+\frac{21}{16}a^2 k^4\sin 3z+O\left(a^3\right),\label{5.3}\\
\varphi_2(z; a, 0)&=\left(\partial_a \eta\right)(z)=\cos z+a k^2\left(\cos 2z-1\right)+\frac{21}{16}a^2 k^4\cos 3z+O\left(a^3\right).\label{5.4}
\end{align}

These functions provide an asymptotic extension for the ($\mu$-independent) basis of the critical subspace $\Sigma_{0, \mu}$ provided in \eqref{5.1}. By spectral perturbation theory, it follows that the functions $\varphi_j(\cdot; a, \mu)$ continue into a $\mu$-dependent basis for the critical subspace $\Sigma_{a, \mu}$ for $|(a, \mu)|\ll 1$. Nevertheless, the discussions in \cite{Hu,Joh13} showed that the variations in the basis functions $\varphi_j(\cdot; a, \mu)$ do not play a role in the asymptotic calculation below as they contribute only to higher-order terms than what are needed here. Thus, in the following, the calculations are done with the $\mu$-independent basis $\varphi_j(\cdot; a, 0)$.

For $\mu$ and $|a|$ sufficiently small, we expand $\mathcal{T}^{\lambda}_{a, \mu}$ using Baker-Campbell-Hausdorff formula as
\begin{equation}\label{5.5}
\mathcal{T}_{a, \mu}^{\lambda}=T_{0, a}+i \mu T_{1, a}-\frac{\mu^2}{2} T_{2, a}+O\left(\mu^3\right),
\end{equation}
as $\mu\rightarrow 0$, where
\begin{align}
T_{0, a}
&:=\mathcal{T}_{a, 0}^{\lambda}=T_0
+2a k^2\left(\sin z\,\partial_z+\partial_{zz}\cos z\right) \notag\\
&\qquad\quad +a^2\left(\frac{k^3}{2\sqrt3}\lambda\partial_z
+2k^4\sin 2z\,\partial_{z}-\frac{3k^4}{2}\partial_{zz}
+k^4\partial_{zz}\cos 2z\right)+O\left(a^3\right),\label{5.6}\\[4pt]
T_{1, a}
&:=\left[\mathcal{T}_{a, 0}^{\lambda},z\right]
=\left[T_0, z\right]+2a k^2\left(\sin z+2\partial_z\cos z\right)\notag\\
&\qquad\qquad\quad +a^2\left(\frac{k^3}{2\sqrt3}\lambda
+2k^4\sin 2z-3k^4\partial_z+2k^4\partial_z\cos 2z\right)+O\left(a^3\right),\label{5.7}\\[4pt]
T_{2,a}
&:=\left[T_{1,a}, z\right]
=\left[\left[T_0, z\right], z\right]
+4a k^2\cos z
+a^2k^4\left(-3+2\cos 2z\right)+O\left(a^3\right),\label{5.8}
\end{align}
and
\begin{equation}\label{5.9}
\begin{aligned}
T_{0}&=\mathcal{T}_{0, 0}^{\lambda}=\frac{2}{\sqrt3 k}\lambda \partial_z-\partial_{zz}-1,\\
T_{1}&=\left[T_0, z\right]=\frac{2}{\sqrt3 k}\lambda-2\partial_z,\\
T_{2}&=\left[\left[T_0, z\right], z\right]=-2.
\end{aligned}
\end{equation}
The detailed calculation for \eqref{5.6}-\eqref{5.9} is provided in Appendix \ref{app B}. Note that $T_{1, a}$ and $T_{2, a}$ are well defined in $L^2(\mathbb{T})$ even though $z$ is not well defined in $L^2(\mathbb{T})$. Moreover, it is easy to get that
\begin{equation*}
\begin{aligned}
T_0\left\{\begin{array}{l}
\cos nz \\
\sin nz
\end{array}\right\}&=\mp\frac{2\lambda}{\sqrt3 k}n\left(\begin{array}{l}
\sin nz \\
\cos nz
\end{array}\right) + \left(n^{2}-1\right)\left(\begin{array}{l}
\cos nz \\
\sin nz
\end{array}\right),\\
T_1\left\{\begin{array}{l}
\cos nz  \\
\sin nz
\end{array}\right\}&=\frac{2\lambda}{\sqrt3 k}\left(\begin{array}{l}
\cos nz\\
\sin nz
\end{array}\right)\pm 2n\left(\begin{array}{l}
\sin nz\\
\cos nz
\end{array}\right),\\
T_2\left\{\begin{array}{c}
\cos nz\\
\sin nz
\end{array}\right\}  &= -2\left(\begin{array}{l}
\cos nz\\
\sin nz
\end{array}\right).\\
\end{aligned}
\end{equation*}
Now, to determine the action of the operator $\mathcal{T}_{a, \mu}^{\lambda}$ in \eqref{5.5} on the generalized kernel $\varphi_j(\cdot; a, 0)$ defined in \eqref{5.3}-\eqref{5.4}, we need to compute the following (the lengthy but straightforward derivation is provided in Appendix \ref{app C}):
\begin{equation}\label{5.9c}
\begin{aligned}
&\mathcal{T}_{a, \mu}^{\lambda} 1= -1-2ak^2\cos z-4a^2k^4\cos 2z+ i\mu\bigg[\frac{2\lambda}{\sqrt{3}k}-2a k^2\sin z+a^2\bigg( \frac{k^3\lambda}{2\sqrt{3}}\quad\quad\quad\\&-2k^4\sin 2z \bigg)\bigg]-\frac{\mu^2}{2}\left[-2+4ak^2\cos z +a^2k^4\left(-3+2\cos 2z\right)\right]+O\left(\mu^3+a^3\right),
\end{aligned}
\end{equation}
\begin{equation}\label{5.9b}
\begin{aligned}
&\mathcal{T}_{a, \mu}^{\lambda} \cos z
= -\frac{2\lambda}{\sqrt{3}\,k}\sin z-a k^2(1+3\cos 2z)+ a^2\left( -\frac{k^3\lambda}{2\sqrt{3}}\sin z- \frac{7k^4}{2}\cos 3z \right) \\
&+ i\mu\left[\frac{2\lambda}{\sqrt{3}k}\cos z+2\sin z-3a k^2\sin 2z+a^2\left( \frac{k^3\lambda}{2\sqrt{3}}\cos z+3k^4\sin z-2k^4\sin 3z \right)\right]
\end{aligned}
\end{equation}
\begin{equation*}
\begin{aligned}
&-\frac{\mu^2}{2}\left[-2\cos z+2ak^2\left(1+\cos 2z\right) +a^2k^4\left(-2\cos z+\cos 3z\right)\right]+O\left(\mu^3+a^3\right),\quad\quad\quad
\end{aligned}
\end{equation*}
\begin{equation}\label{5.9a}
\begin{aligned}
&\mathcal{T}_{a, \mu}^{\lambda} \sin z= \frac{2\lambda}{\sqrt{3}k}\cos z- 3a k^2\sin 2z+ a^2\left(\frac{k^3\lambda}{2\sqrt{3}}\cos z+ 3k^4\sin z- \frac{7k^4}{2}\sin 3z \right)\quad\\
&+ i\mu\bigg[\frac{2\lambda}{\sqrt{3}k}\sin z- 2\cos z+a k^2(1+3\cos2z)+a^2\bigg( \frac{k^3\lambda}{2\sqrt{3}}\sin z- 3k^4\cos z+ 2k^4\\&\cos 3z \bigg)\bigg]-\frac{\mu^2}{2}\left[-2\sin z+2ak^2\sin 2z +a^2k^4(-4\sin z+\sin 3z)\right]+O\left(\mu^3+a^3\right),
\end{aligned}
\end{equation}
\begin{equation}\label{5.9d}
\begin{aligned}
&\mathcal{T}_{a, \mu}^{\lambda} \cos 2z= -\frac{4\lambda}{\sqrt{3}k}\sin 2z+3\cos2z-a k^2(3\cos z+7\cos 3z)+a^2\bigg(-\frac{k^3\lambda}{\sqrt{3}}\sin 2z\\
&+6k^4\cos 2z-2k^4- 6k^4\cos 4z \bigg) + i\mu\left[\frac{2\lambda}{\sqrt{3}k}\cos 2z+4\sin 2z-a k^2\left(3\sin z\right.\right.\\
&\left.\left.+5\sin 3z\right)+a^2\left(\frac{k^3\lambda}{2\sqrt{3}}\cos 2z+6k^4\sin 2z-3k^4\sin 4z \right)\right]-\frac{\mu^2}{2}\left[-2\cos 2z\right.\\
&\left.+2ak^2\left(\cos z+\cos 3z\right)+a^2k^4\left(1-3\cos 2z+\cos 4z\right)\right]+O\left(\mu^3+a^3\right),
\end{aligned}
\end{equation}
\begin{equation}\label{5.9e}
\begin{aligned}
&\mathcal{T}_{a, \mu}^{\lambda} \sin 2z
= \frac{4\lambda}{\sqrt{3}k}\cos 2z + 3\sin 2z - a k^2\left(3\sin z + 7\sin 3z\right) + a^2\bigg(\frac{k^3\lambda}{\sqrt{3}}\cos 2z \\&+ 6k^4\sin 2z- 6k^4\sin 4z \bigg)+ i\mu\left[\frac{2\lambda}{\sqrt{3}k}\sin 2z - 4\cos 2z + a k^2\left(3\cos z + 5\cos 3z\right)\right.\\&\left.+ a^2\left(\frac{k^3\lambda}{2\sqrt{3}}\sin 2z - 6k^4\cos 2z+ 3k^4\cos 4z + k^4 \right)\right]- \frac{\mu^2}{2}\left[-2\sin 2z + 2ak^2\left(\sin z \right.\right.\\&\left.\left.+ \sin 3z\right)+ a^2k^4\left(-3\sin 2z + \sin 4z\right)\right]+O\left(\mu^3+a^3\right).
\end{aligned}
\end{equation}

Using these and \eqref{5.3}-\eqref{5.4}, we get
\begin{align*}
\left\langle\mathcal{T}_{a, \mu}^{\lambda} \varphi_1, \varphi_1\right\rangle
&= \left\langle\mathcal{T}_{a, \mu}^{\lambda} \left(\sin z+a k^2 \sin 2z+\frac{21}{16}a^2 k^4\sin 3z+O\left(a^3\right)\right),\right.\\
&\qquad\left.\sin z+a k^2 \sin 2z+\frac{21}{16}a^2 k^4\sin 3z+O\left(a^3\right)\right\rangle\\
&= i \mu \left(\frac{\lambda}{\sqrt3 k}+\frac{5 \lambda }{4\sqrt3}a^2k^3\right)+\frac{\mu^2}{2}\left(1+a^2k^4\right)+O\left(\mu^3+a^3\right),\\[6pt]
\left\langle\mathcal{T}_{a, \mu}^{\lambda} \varphi_1, \varphi_2\right\rangle
&= \left\langle\mathcal{T}_{a, \mu}^{\lambda} \left(\sin z+a k^2 \sin 2z+\frac{21}{16}a^2 k^4\sin 3z+O\left(a^3\right)\right),\right.\\
&\qquad\left. \cos z+a k^2 (\cos 2z-1)+\frac{21}{16}a^2 k^4\cos 3z+O\left(a^3\right)\right\rangle\\
&= \frac{\lambda}{\sqrt3 k}+ \frac{9\lambda}{4\sqrt3 }a^2k^3-i \mu(1+\frac{3}{2}a^2k^4)+O\left(\mu^3+a^3\right),\\[6pt]
\left\langle\mathcal{T}_{a, \mu}^{\lambda} \varphi_2, \varphi_1\right\rangle
&= \left\langle\mathcal{T}_{a, \mu}^{\lambda} \left(\cos z+a k^2 (\cos 2z-1)+\frac{21}{16}a^2 k^4\cos 3z+O\left(a^3\right)\right),\right.\\
&\qquad\left. \sin z+a k^2 \sin 2z+\frac{21}{16}a^2 k^4\sin 3z+O\left(a^3\right)\right\rangle\\
&= -\frac{\lambda}{\sqrt3k}-\frac{9\lambda}{4\sqrt3}a^2k^3+i \mu\left(1+\frac{3}{2}a^2k^4\right)+O\left(\mu^3+a^3\right),\\[6pt]
\left\langle\mathcal{T}_{a, \mu}^{\lambda} \varphi_2, \varphi_2\right\rangle
&= \left\langle\mathcal{T}_{a, \mu}^{\lambda} \left(\cos z+a k^2 (\cos 2z-1)+\frac{21}{16}a^2 k^4\cos 3z+O\left(a^3\right)\right),\right.\\
&\qquad\left. \cos z+a k^2 (\cos 2z-1)+\frac{21}{16}a^2 k^4\cos 3z+O\left(a^3\right)\right\rangle\\
&= -\frac{1}{2}a^2k^4+i \mu\left(\frac{\lambda}{\sqrt3 k}+\frac{13\lambda}{4\sqrt3}a^2k^3\right)+\frac{\mu^2}{2}\left(1+6a^2 k^4\right)+O\left(\mu^3+a^3\right).
\end{align*}
Moreover, it is easy to obtain
$$
\left\langle\varphi_1, \varphi_1\right\rangle =\frac{1+a^2k^4}{2}+O\left(a^3\right),
\left\langle\varphi_2, \varphi_2\right\rangle=\frac{1+3a^2k^4}{2}+O\left(a^3\right),\left\langle\varphi_1, \varphi_2\right\rangle =\left\langle\varphi_2, \varphi_1\right\rangle=0,
$$
as $ a\rightarrow 0$.
Using the above obtained expressions, it follows that the matrix $\mathbf{B}^{\lambda}_{a, \mu}$ in \eqref{5.2} can be expanded for sufficiently small $\mu$ and $|a|$ as
\begin{equation}\label{5.10}
\begin{aligned}
&\mathbf{B}_{a, \mu}^{\lambda}=\left(\begin{array}{cc}
0 & \frac{2\lambda}{\sqrt3 k}  \\
-\frac{2\lambda}{\sqrt3 k} & 0
\end{array}\right)+a^2\left(\begin{array}{cc}
0 & \frac{5\lambda}{2\sqrt3}k^3\\
\frac{3\lambda}{2\sqrt3}k^3 & -k^4
\end{array}\right)+2 i \mu\left(\begin{array}{cc}
\frac{\lambda}{\sqrt3 k} & -1\\
1 & \frac{\lambda}{\sqrt3 k}
\end{array}\right)\\&+i \mu a^2\left(\begin{array}{cc}
\frac{\lambda}{2\sqrt3}k^3 & -k^4\\
-3k^4 & \frac{\lambda}{2\sqrt3}k^3
\end{array}\right)+\mu^2\left(\begin{array}{cc}
1 & 0\\
0 & 1
\end{array}\right)+\mu^2 a^2\left(\begin{array}{cc}
0 & 0\\
0 & 3k^4
\end{array}\right)+O\left(\mu^3+a^3\right).
\end{aligned}
\end{equation}
The attention now is turned to the roots of the characteristic polynomial
\begin{equation}\label{5.11}
\begin{aligned}
&\det(\mathbf{B}^{\lambda}_{a, \mu})
=\left(-4\mu^2 + 3a^2 k^4 \mu^2 + 3a^4 k^8 \mu^2 + \mu^4 + 3a^2 k^4 \mu^4\right) \\
&+i \left(-\frac{8\mu}{\sqrt{3}k} + \frac{17 a^4 k^7 \mu}{2\sqrt{3}} + \frac{4\mu^3}{\sqrt{3}k} + \frac{7 a^2 k^3 \mu^3}{\sqrt{3}} + \frac{\sqrt{3}}{2}a^4 k^7 \mu^3\right)\lambda \\
&+\left(\frac{4}{3k^2} + \frac{2a^2k^2}{3} - \frac{5a^4k^6}{4} - \frac{4\mu^2}{3k^2} - \frac{2}{3}a^2k^2\mu^2 - \frac{1}{12}a^4k^6\mu^2\right)\lambda^2\\&:=b_0(a, \mu)+ ib_1(a, \mu) \lambda+b_2(a, \mu) \lambda^2,
\end{aligned}
\end{equation}
where the coefficient functions $b_j$ ($j=0, 1, 2$) are real-valued functions and depend smoothly on the parameters $a$ and $\mu$. In particular, $b_0$ and $b_2$ are even functions in $\mu$, while $b_1$ is an odd function in $\mu$. We further note that $b_j(a, \mu)=d_j(a, \mu) \mu^{2-j}, j=0, 1, 2$, where the coefficient functions $d_j$ ($j=0, 1, 2$) are real-valued functions and depend smoothly on $a$ and $\mu$ for $|(a, \mu)|\ll1$. Setting $\lambda=i \mu X$, it follows that
$$
\det(\mathbf{B}^{\lambda}_{a, \mu})=\mu^2\left(d_0(a, \mu)-d_1(a, \mu)X-d_2(a, \mu)X^2\right)=:\mu^2 Q(a, \mu, X).
$$
Therefore, the underlying wave is modulationally unstable if the polynomial $Q$ possesses roots with non-zero imaginary parts, and modulationally stable if $Q$ admits two distinct real roots.

To determine the reality of these roots, it suffices to analyze the discriminant $\mathcal{D}_{a, \mu}$ of $Q$. Consequently, the asymptotically small background periodic traveling waves $\eta(\cdot; a, k)$ are modulationally stable provided that $\mathcal{D}_{a, \mu}>0$ for $0<|\mu|\ll1$ and modulationally unstable if $\mathcal{D}_{a, \mu}<0$ for $0<|\mu|\ll1$. Via a Mathematica calculation, the discriminant $\mathcal{D}_{a, \mu}$ can be directly expanded as
$$
\mathcal{D}_{a, \mu}=\frac{16 \mu^2}{3 k^2}+\frac{16 a^2 k^2}{3}+O\left(a^2(\mu^2+a^2)\right).
$$
We note that $\mathcal{D}_{a, \mu}>0$ for all $0<|\mu|\ll1$ and sufficiently small $|a|$, which implies the modulational stability of the background wave. This completes the proof of Theorem \ref{the5.1}.
\end{proof}
\begin{remark}\label{re1}
In a recent paper \cite{Locke}, the authors showed that a $2 \pi$-periodic wave of the local shallow water wave model \eqref{1.1} (taking $k=1$ in the present paper) is spectrally stable in the subspace $L_0^2(\mathbb{T})=\left\{f \in L^2(\mathbb{T}): \int_0^{2 \pi} f(z) d z=0\right\}$ of $L^2(\mathbb{T})$ with respect to co-periodic perturbations $(\mu=0)$. This result agrees with what we obtained in the present paper. Specifically, taking $k=1$ and $\mu=0$, we observe from \eqref{5.11} that $\det(\mathbf{B}^{\lambda}_{a, \mu})=\lambda^2(\frac{4}{3}+\frac{2a^2}{3}-\frac{5a^4}{4})$ and thus zero is a root of multiplicity two. Then the periodic traveling wave is spectrally stable under co-periodic perturbations. Furthermore, for the case of the long-wavelength perturbations $(0<\mu\leq1)$, we get from Theorem \ref{the5.1} that a sufficiently small $2 \pi$-periodic wave of \eqref{1.1} is modulationally stable.
\end{remark}
\begin{remark}\label{re2}
In a recent paper \cite{Bhavna}, the authors considered the dispersion-generalized Ostrovsky (gOst) equation
\begin{equation}\label{5.13}
\left(v_t+\beta\mathcal{M}v_x+(v^2)_x\right)_x=\alpha v,\quad \alpha>0, \; \beta\in\mathbb{R} \setminus \{0\},
\end{equation}
where $\mathcal{M}$ is a Fourier multiplier operator required to satisfy Hypothesis 1.1 proposed therein. By taking $\beta = -c^2$ and $\mathcal{M} = 1$--noting that $\mathcal{M}=1$ readily satisfies Hypothesis 1.1--the gOst equation reduces to
\begin{equation}\label{5.12}
\left(v_t-c^2v_x+(v^2)_x\right)_x=\alpha v,\quad \alpha>0.
\end{equation}
For this special case, we obtain from \cite{Bhavna} that
\begin{equation*}
c_p(k)=\beta m(k)+\frac{\alpha}{k^2}=-c^2+\frac{\alpha}{k^2},\quad
c_g(k)=\beta (m(k)+k m^{\prime}(k))-\frac{\alpha}{k^2}=-c^2-\frac{\alpha}{k^2},
\end{equation*}
which yields that
\begin{equation*}
c^{\prime}_g(k)=2\frac{\alpha}{k^3}>0,\quad
c_p(k)-c_p(2k)=\frac{3\alpha}{4k^2}>0,
\end{equation*}
and hence the modulational instability index
\begin{equation*}
\Delta(k)=(c_p(k)-c_p(2k))\frac{dc_g(k)}{dk}>0.
\end{equation*}
Theorem 1.3 in \cite{Bhavna} then implies that a $2 \pi / k$-periodic traveling wave of \eqref{5.12} with sufficiently small amplitude is modulationally stable. Comparing \eqref{5.12} with the reformulation of the local model \eqref{1.1}, namely \eqref{1.4}, we find that the main difference lies in the presence of the term $(u_x)^2$ on the right-hand side of \eqref{1.4}. The above discussion together with our main result (Theorem \ref{the5.1}), reveals that the inclusion of the term $(u_x)^2$ does not affect the modulational stability.
\end{remark}
\section{Another extended Hunter-Saxton equation}\label{Another model}
In this section, we analyze the modulational instability of the extended Hunter-Saxton equation \eqref{6.1}. We first follow a procedure similar to that in Section \ref{sec expression} to obtain the existence for periodic traveling waves of \eqref{6.1}, and their small-amplitude expansion can also be given as follows.
\begin{lemma}\label{le6}
For all wavenumbers $k>0$, there exists a one parameter family of solutions of \eqref{6.1} given by $u(x, t)=w(a; k)\left(k(x-c(a; k)t)\right)$ for $a\in \mathbb{R}$ and $|a|$ sufficiently small; $w(a; k)(\cdot)$ is $2 \pi$-periodic, even and smooth in its argument, and $c(a; k)$ is even in $a$; $w(a; k)(\cdot)$ and $c(a; k)$ depend analytically on $a$ and $k$. Moreover, as $a \to 0$,
\begin{equation}\label{6.2}
w(a; k)(z)= a \cos(z)+a^{2}\left(A_0+A_2\cos(2z)\right)+a^{3} A_3\cos(3z)+O(a^4),
\end{equation}
\begin{equation}\label{6.3}
c(a; k)=c_0+a^2 c_2+O(a^4),
\end{equation}
where
\begin{equation}\label{6.4}
\begin{array}{ll}
\displaystyle A_0=-\frac{k^2}{4},\quad A_2=-A_0,\quad A_3=\frac{(7+\gamma)k^4}{64}, \\
\displaystyle c_0=\frac{1}{k^2}, \quad c_2=\frac{(1-\gamma)k^2}{8}.
\end{array}
\end{equation}
\end{lemma}
Linearizing equation \eqref{6.1} about $w$ given in \eqref{6.2}, 
we obtain the linear evolution equation
\begin{equation}\label{6.5}
k v_{t z}-ck^2v_{z z}+k^2(w v)_{z z}-k^2w_z v_z-\gamma k^4 w_z w_{z z} v_z-\frac{\gamma}{2}k^4(w_z)^2v_{z z}-v=0.
\end{equation}
For $v(z, t)=\mathrm{e}^{\frac{\lambda}{k}t} \widetilde{v}(z)$, $\lambda\in\mathbb{C}$ and $\widetilde{v}(z) \in L^2(\mathbb{R})$, we arrive at
\begin{equation}\label{6.6}
\mathcal{A}^{\lambda}_{k, a} \widetilde{v}:= \left[(\lambda-k^2 w_z-\gamma k^4w_z w_{z z})\p_z+k^2 \p_{z z}\left(-c+w\right)-\frac{\gamma k^4}{2}(w_z)^2 \p_{z z}-1\right]\widetilde{v}=0,
\end{equation}
where $\mathcal{A}^{\lambda}_{k, a}:H^2(\mathbb{R})\rightarrow L^2(\mathbb{R})$ is considered as a closed, densely defined linear operator. Using Floquet theory, the invertibility problem of the operator $\mathcal{A}^{\lambda}_{k, a}$ on $L^2(\mathbb{R})$ reduces to a one-parameter family of invertibility problem of Bloch operators on $L^2(\mathbb{T})$. The linear operator $\mathcal{A}^{\lambda}_{k, a}$ is invertible with bounded inverse on $L^2(\mathbb{R})$ if and only if the Bloch operators
\begin{equation}\label{6.7}
\begin{aligned}
\mathcal{A}^{\lambda}_{k, a, \mu} :&=\left(\lambda-k^2 w_z-\gamma k^4w_z w_{z z}\right)\left(\p_z+i \mu\right)+k^2\left(\p_{z}+i\mu\right)^2\left(-c+w\right)\\&-\frac{\gamma k^4}{2}(w_z)^2\left(\p_{z}+i\mu\right)^2-1,
\end{aligned}
\end{equation}
acting in $L^2(\mathbb{T})$ with domain $H^2(\mathbb{T})$ are invertible for all $\mu \in(-\frac{1}{2}, \frac{1}{2}]$. One can easily establish the symmetry property
$$\mathcal{A}^{\lambda}_{k, a, \mu}(z)=\overline{\mathcal{A}^{-\overline{\lambda}}_{k, a, \mu}(-z)}=\overline{\mathcal{A}^{\overline{\lambda}}_{k, a, -\mu}(z)},$$
which implies that it suffices to consider $\mu \in\left[0, \frac{1}{2}\right]$ when studying the $L^2(\mathbb{T})$-kernel of the operators $\mathcal{A}^{\lambda}_{k, a, \mu}$, and that the set of $\lambda$ where these operators fail to be invertible is symmetric with respect to reflections about the imaginary axis. Since $k$ is fixed, we denote $\mathcal{A}^{\lambda}_{k, a, \mu}$ by $\mathcal{A}^{\lambda}_{a, \mu}$.

A direct computation yields that for all $n\in \mathbb{Z}$ and $\mu \in \left[0, \frac{1}{2}\right]$, the kernel of $\mathcal{A}^{\lambda}_{0, \mu}$ is non-trivial when
\begin{equation}\label{6.8}
\lambda=i\left(n+\mu-\frac{1}{n+\mu}\right)=\frac{2}{\sqrt{3} k}i\Omega_{n, 0, \mu},
\end{equation}
where $\Omega_{n, 0, \mu}$ is defined in \eqref{4.2}. Then the collisions of $\lambda$  are consistent with those described in Lemma \ref{lem4.1}, with all collisions occurring away from the origin in the complex plane, except for the case $i\Omega_{-1, 0, 0}=i\Omega_{1, 0, 0}=0$. On the other hand, we can also get
$$
\left\|\mathcal{A}^{\lambda}_{a, \mu}-\mathcal{A}^{\lambda}_{0, 0}\right\|_{H^2(\mathbb{T})\rightarrow L^2(\mathbb{T})}=O(|a|+|\mu|),
$$
which holds uniformly in operator norm as $|a|, \mu \rightarrow 0$. Accordingly, for $|(a, \mu)|\ll1$ there can be only two values $\lambda=i \Omega_{\pm 1, a, \mu}$ in a sufficiently small neighborhood of the origin such that the operator $\mathcal{A}^{\lambda}_{a, \mu}$ is non-invertible. Then, proceeding similarly to the proof of Theorem \ref{the5.1}, we obtain the expansion of the discriminant $\mathcal{D}_{a, \mu}$ as
\begin{equation}\label{6.9}
\mathcal{D}_{a, \mu}=4 \mu^2+ a^2k^4\left(1-\gamma\right)+O\left(a^2\left(\mu^2+a^2\right)\right),
\end{equation}
and a parallel discussion gives the following result.
\begin{theorem}\label{the6.1}
A sufficiently small $2\pi/k$-periodic traveling wave of \eqref{6.1} is modulationally unstable if $\gamma >1$, while it is modulationally stable if $\gamma <1$.
\end{theorem}

\appendix
\section{Small-amplitude expansion}\label{app A}
In this section, we give the details on small-amplitude expansion of \eqref{2.1}. For a fixed $k>0$, the solutions $\big(\eta(a; k)(\cdot), c(a, k)\big)$ are analytic in $a$ for $|a|\ll1$ and hence can be expanded as
\begin{equation}\label{2.7}
\left\{
\begin{aligned}
&\eta(a;k)(z)= a \cos(z)+a^{2}w_2(z)+a^{3} w_3(z)+O(a^4),\\
&c(a;k)=c_0(k)+c_2(k)a^2 +O(a^4).
\end{aligned}
\right.
\end{equation}
Substituting $\eta$ and $c$ from \eqref{2.7} into \eqref{2.1}, and expanding in terms of the order of $a$, we obtain
\begin{equation}\label{2.8}
\begin{aligned}
&\left(-3c_0^{2}k^2 \cos(z)+\cos(z)\right)a=0,\\
&\left(3c_0^{2} k^2w_2^{\prime \prime}+2k^2\cos^{2}(z)-k^2\sin^{2}(z)+w_2\right)a^2=0, \\
&\left(3c_0^{2} k^2w_3^{\prime \prime}-6c_0c_2k^2\cos(z)-2k^2\cos(z)\left(w_2^{\prime \prime}-w_2\right)+2k^2\sin(z)w_2^{\prime}+w_3\right)a^3=0,
\end{aligned}
\end{equation}
where the first identity in \eqref{2.8} follows directly from \eqref{2.3a} and we have
\begin{equation}\label{2.9}
\begin{array}{ll}
\displaystyle w_2^{\prime \prime}+w_2+\frac{k^2}{2}+\frac{3k^2}{2}\cos(2z)=0, \\
\displaystyle w_3^{\prime \prime}+w_3-2\sqrt3 c_2 k\cos(z)-2k^2\cos(z)\left(w_2^{\prime \prime}-w_2\right)+2k^2\sin(z)w_2^{\prime}=0.
\end{array}
\end{equation}
Solving the above equations yields
\begin{equation}\label{2.10}
c_2=\frac{k^3}{4\sqrt3}, \quad w_2=\frac{k^2}{2}\left(\cos(2z)-1\right), \quad w_3=\frac{7k^4}{16}\cos(3z).
\end{equation}
\section{Calculation of equations \eqref{5.6}-\eqref{5.9}}\label{app B}
From $\mathcal{T}^{\lambda}_{k, a, \mu}$ given in \eqref{3.2}, we obtain
\begin{equation}\label{B.1}
\mathcal{T}^{\lambda}_{a, 0}=2c \lambda\p_z-2k^2 \eta_z\p_z -3c^2k^2\p_{zz}+2k^2\p_{zz}\eta-1.
\end{equation}
Substituting \eqref{2.4}-\eqref{2.6} into \eqref{B.1}, we observe that for any $f\in H^2(\mathbb{T})$
\begin{equation}\label{B.2}
\begin{aligned}
T_{0, a}f&:=\mathcal{T}^{\lambda}_{a, 0}f=\left(2c \lambda\p_z-2k^2 \eta_z\p_z -3c^2k^2\p_{zz}+2k^2\p_{zz}\eta-1\right)f\\&
=\left(2\left(\frac{1}{\sqrt3 k}+\frac{k^3 a^2}{4\sqrt3}\right)\lambda\p_z-2k^2\left(-a\sin z-a^2k^2\sin 2z\right)\p_z-3\left(\frac{1}{3k^2}\right.\right.\\
&\left.\left.+\frac{k^2a^2}{6}\right)k^2\p_{zz}+2k^2\p_{zz}\left(a\cos z+\frac{a^2k^2}{2}\cos 2z-\frac{a^2k^2}{2}\right)-1\right)f+O\left(a^3\right)\\&
=\left(\frac{2}{\sqrt3 k}\lambda\p_z-\p_{zz}-1+2ak^2(\sin z\p_{z}+\p_{zz}\cos z)\right.\\
&\left.+a^2(\frac{k^3}{2\sqrt3}\lambda\p_z+2k^4\sin 2z\p_z-\frac{3k^4}{2}\p_{zz}+k^4\p_{zz}\cos 2z)\right)f+O\left(a^3\right)\\
&=:T_0 f+\left(2ak^2(\sin z\p_{z}+\p_{zz}\cos z)+a^2\left(\frac{k^3}{2\sqrt3}\lambda\p_z+2k^4\sin 2z\p_z\right.\right.\\
&\left.\left.-\frac{3k^4}{2}\p_{zz}+k^4\p_{zz}\cos 2z\right)\right)f+O\left(a^3\right),
\end{aligned}
\end{equation}
where $T_0=\frac{2}{\sqrt3 k}\lambda\p_z-\p_{zz}-1$. This gives the equation \eqref{5.6}. To get \eqref{5.7}, we find that
\begin{equation}\label{B.3}
\begin{aligned}
T_1 f&:=\left[T_0, z\right]f
=\left[\frac{2}{\sqrt3 k}\lambda\partial_z-\partial_{zz}-1, z\right]f
=\left(\frac{2\lambda}{\sqrt3 k}-2\partial_z\right)f,\quad
\end{aligned}
\end{equation}
\begin{equation}\label{B.4}
\begin{aligned}
&\left[\left(2ak^2\left(\sin z\p_{z}+\p_{zz}\cos z\right)+a^2\left(\frac{k^3}{2\sqrt3}\lambda\p_z+2k^4\sin 2z\p_z-\frac{3k^4}{2}\p_{zz}+k^4\p_{zz}\cos 2z\right)\right),z\right]f\\&
=2ak^2\left[\sin z\p_{z}+\p_{zz}\cos z,z\right]f+a^2\left[\frac{k^3\lambda}{2\sqrt3}\p_z+2k^4\sin 2z\p_z-\frac{3k^4}{2}\p_{zz}+k^4\p_{zz}\cos 2z,z\right]f\\
&=2ak^2\left(\sin z\p_{z}(zf)-z\sin z\p_{z}f+\p_{zz}(z f\cos z)-z\p_{zz}(f\cos z)\right)\\
&+a^2\left(\left[(\frac{k^3\lambda}{2\sqrt3}+2k^4\sin 2z)\p_z,z\right]f-\frac{3k^4}{2}\left[\p_{zz},z\right]f+k^4\left[\p_{zz}\cos 2z,z\right]f\right)\\&
=\left(2ak^2\left(\sin z+2\p_{z}\cos z\right)+a^2\left(\frac{k^3\lambda}{2\sqrt3}+2k^4\sin 2z-3k^4\p_z+2k^4\p_z\cos 2z\right)\right)f.
\end{aligned}
\end{equation}
Then \eqref{B.2}-\eqref{B.4} give \eqref{5.7}. To obtain \eqref{5.8}, we find that
\begin{equation}\label{B.5}
T_2 f:=\left[T_1, z\right]f
=\left[\frac{2\lambda}{\sqrt3 k}-2\p_z, z\right]f
=-2f,
\end{equation}
\begin{equation}\label{B.6}
\begin{aligned}
&\left[2ak^2\left(\sin z+2\p_{z}\cos z\right)+a^2\left(\frac{k^3\lambda}{2\sqrt3}+2k^4\sin 2z-3k^4\p_z+2k^4\p_z\cos 2z\right),z\right]f\\
&=2ak^2\left[\sin z+2\partial_{z}\cos z,z\right]f+a^2\bigg(\left[\frac{k^3\lambda}{2\sqrt3}+2k^4\sin 2z,z\right]f-3k^4\left[\partial_z,z\right]f\\
&\quad+2k^4\left[\partial_z\cos 2z, z\right]\bigg)f\\
&=2ak^2\cdot 2\left(\partial_z(zf\cos z)-z\partial_z(f\cos z)\right)+a^2\left(-3k^4f+2k^4f\cos 2z\right)\\
&=\left(4ak^2\cos z+a^2k^4\left(-3+2\cos 2z\right)\right)f.
\end{aligned}
\end{equation}
Then \eqref{B.5} and \eqref{B.6} give \eqref{5.8}.
\section{Calculation of equations \eqref{5.9c}-\eqref{5.9e}}\label{app C}
From $T_{0,a}, T_{1,a}, T_{2,a}$ given in \eqref{5.6}-\eqref{5.8}, we obtain that
\begin{equation*}
\begin{aligned}
T_{0,a} 1
&=\frac{2}{\sqrt3 k}\lambda \p_z 1-\p_{zz}1-1
+2a k^2\left(\sin z\,\partial_z 1+\partial_{zz}\cos z\right) \notag\\
&+a^2\left(\frac{k^3}{2\sqrt3}\lambda\partial_z 1
+2k^4\sin 2z\,\partial_{z} 1-\frac{3k^4}{2}\partial_{zz} 1
+k^4\partial_{zz}\cos 2z\right)+O\left(a^3\right)\\
&=-1-2ak^2\cos z-4a^2k^4\cos 2z+O\left(a^3\right),
\end{aligned}
\end{equation*}
\begin{equation*}
\begin{aligned}
T_{1,a} 1
&=\frac{2}{\sqrt3 k}\lambda-2\p_z1+2a k^2\left(\sin z+2\partial_z\cos z\right)\notag\\
&+a^2\left(\frac{k^3}{2\sqrt3}\lambda
+2k^4\sin 2z-3k^4\partial_z 1+2k^4\partial_z\cos 2z\right)+O\left(a^3\right)\quad\quad\quad\quad\\&
=\frac{2\lambda}{\sqrt{3}k}-2a k^2\sin z+a^2\left(\frac{k^3\lambda}{2\sqrt{3}}-2k^4\sin 2z \right)+O\left(a^3\right),
\end{aligned}
\end{equation*}
\begin{equation*}
\begin{aligned}
T_{2,a} 1
&=-2+4ak^2\cos z +a^2k^4\left(-3+2\cos 2z\right)+O\left(a^3\right),\quad\quad\quad\quad\quad\quad\quad
\end{aligned}
\end{equation*}
then according to equation \eqref{5.5}, we have
\begin{equation}\label{C.1}
\begin{aligned}
\mathcal{T}_{a, \mu}^{\lambda} 1&=T_{0, a}1+i \mu T_{1, a}1-\frac{\mu^2}{2} T_{2, a}1+O\left(\mu^3\right)\\&
= -1-2ak^2\cos z-4a^2k^4\cos 2z\\
&+ i\mu\left[\frac{2\lambda}{\sqrt{3}k}-2a k^2\sin z+a^2\left( \frac{k^3\lambda}{2\sqrt{3}}-2k^4\sin 2z \right)\right]\quad\quad\quad\quad\quad\quad\quad\\&-\frac{\mu^2}{2}\left[-2+4ak^2\cos z +a^2k^4(-3+2\cos 2z)\right]+O\left(\mu^3+a^3\right).
\end{aligned}
\end{equation}
Similarly,
\begin{equation*}
\begin{aligned}
T_{0,a} \cos z
&= \frac{2}{\sqrt{3}k}\lambda \partial_z \cos z - \partial_{zz} \cos z - \cos z + 2ak^2\left(\sin z \partial_z \cos z + \partial_{zz} \cos^2 z\right)+ a^2\\&\left( \frac{k^3}{2\sqrt{3}}\lambda \partial_z \cos z + 2k^4 \sin 2z \partial_z \cos z - \frac{3k^4}{2}\partial_{zz} \cos z + k^4 \partial_{zz}\left(\cos 2z \cos z\right) \right)+O\left(a^3\right) \\
&= -\frac{2}{\sqrt{3}k}\lambda \sin z + \cos z - \cos z + 2ak^2\left(-\sin^2 z + 2\sin^2 z - 2\cos^2 z\right)\\
&+ a^2\bigg(-\frac{k^3}{2\sqrt{3}}\lambda \sin z - 2k^4 \sin 2z \sin z + \frac{3k^4}{2}\cos z + k^4\left(-4\cos 2z \cos z\right.\\
&\quad \left.+ 4\sin 2z \sin z - \cos 2z \cos z\right)\bigg)+O\left(a^3\right)\\&
=-\frac{2}{\sqrt{3}k}\lambda \sin z +2ak^2\left(\sin^2 z - 2\cos^2 z\right) + a^2\bigg( -\frac{k^3}{2\sqrt{3}}\lambda \sin z+ \frac{3k^4}{2}\cos z\\
&\quad\quad - 2k^4 \sin z \sin 2z+ k^4\left( 4\sin z \sin 2z-5\cos 2z \cos z\right)\bigg)+O\left(a^3\right)\\
&=-\frac{2}{\sqrt{3}k}\lambda \sin z +2ak^2\left(\frac{1-\cos 2z}{2}-2\frac{1+\cos 2z}{2}\right)+ a^2\left(-\frac{k^3}{2\sqrt{3}}\lambda \sin z\right.\\
&\left.+\frac{3k^4}{2}\cos z +2k^4 \sin z \sin 2z-5k^4\cos 2z \cos z\right)+O\left(a^3\right)\\
&=-\frac{2}{\sqrt{3}k}\lambda \sin z +2ak^2\left(-\frac{1}{2}-\frac{3}{2}\cos 2z\right)+ a^2\left( -\frac{k^3}{2\sqrt{3}}\lambda \sin z+ \frac{3k^4}{2}\cos z\right.\\
&\left.\quad\quad +2k^4\frac{\cos z-\cos3z}{2}-5k^4\frac{\cos z+\cos3z}{2}\right)+O\left(a^3\right)\\
&=-\frac{2\lambda}{\sqrt{3}\,k}\sin z-a k^2\left(1+3\cos 2z\right)+ a^2\left( -\frac{k^3\lambda}{2\sqrt{3}}\sin z- \frac{7k^4}{2}\cos 3z \right)+O\left(a^3\right),
\end{aligned}
\end{equation*}
\begin{equation*}
\begin{aligned}
T_{1,a} \cos z
&=\frac{2}{\sqrt3 k}\lambda\cos z-2\p_z\cos z+2a k^2\left(\sin z\cos z+2\partial_z\cos^2 z\right)+a^2\bigg(\frac{k^3}{2\sqrt3}\lambda\cos z
\\&+2k^4\sin 2z\cos z-3k^4\partial_z \cos z+2k^4\partial_z\left(\cos 2z\cos z\right)\bigg)+O\left(a^3\right)\\&
=\frac{2\lambda}{\sqrt{3}k}\cos z+2\sin z+2a k^2\left(\sin z\cos z-4\cos z\sin z\right)+a^2\bigg(\frac{k^3\lambda}{2\sqrt{3}}\cos z\quad\quad\quad\\
&+2k^4\sin 2z\cos z+3k^4\sin z+2k^4\left(-2\sin 2z \cos z-\cos 2z \sin z\right)\bigg)+O\left(a^3\right)
\end{aligned}
\end{equation*}
\begin{equation*}
\begin{aligned}
\quad\quad\quad&=\frac{2\lambda}{\sqrt{3}k}\cos z+2\sin z-6a k^2\sin z\cos z+a^2\bigg( \frac{k^3\lambda}{2\sqrt{3}}\cos z+3k^4\sin z\\
&-2k^4\sin 2z\cos z-2k^4\cos 2z \sin z\bigg)+O\left(a^3\right)\\&
=\frac{2\lambda}{\sqrt{3}k}\cos z+2\sin z-3a k^2\sin 2z+a^2\left( \frac{k^3\lambda}{2\sqrt{3}}\cos z+3k^4\sin z\right.\\
&\left.-2k^4\frac{\sin 3z+\sin z}{2}-2k^4\frac{\sin 3z-\sin z}{2}\right)+O\left(a^3\right)\\
&=\frac{2\lambda}{\sqrt{3}k}\cos z+2\sin z-3a k^2\sin 2z+a^2\bigg( \frac{k^3\lambda}{2\sqrt{3}}\cos z+3k^4\sin z\\&-2k^4\sin 3z \bigg)+O\left(a^3\right),
\end{aligned}
\end{equation*}
\begin{equation*}
\begin{aligned}
T_{2,a} \cos z
&=-2\cos z+4a k^2\cos^2 z+a^2k^4\left(-3\cos z+2\cos 2z\cos z\right)+O\left(a^3\right)\\&
=-2\cos z+4ak^2\frac{1+\cos 2z}{2} +a^2k^4\left(-3\cos z+2\frac{\cos 3z+\cos z}{2}\right)+O\left(a^3\right)\quad\quad\\&
=-2\cos z+2ak^2\left(1+\cos 2z\right) +a^2k^4\left(-2\cos z+\cos 3z\right)+O\left(a^3\right),
\end{aligned}
\end{equation*}
then we have
\begin{equation}\label{C.2}
\begin{aligned}
\mathcal{T}_{a, \mu}^{\lambda} \cos z&=T_{0, a}\cos z+i \mu T_{1, a}\cos z-\frac{\mu^2 }{2} T_{2, a}\cos z+O\left(\mu^3\right)\\&
=-\frac{2\lambda}{\sqrt{3}\,k}\sin z-a k^2\left(1+3\cos 2z\right)+ a^2\left( -\frac{k^3\lambda}{2\sqrt{3}}\sin z- \frac{7k^4}{2}\cos 3z \right)\quad\quad\\
&+ i\mu\left[\frac{2\lambda}{\sqrt{3}k}\cos z+2\sin z-3a k^2\sin 2z+a^2\bigg( \frac{k^3\lambda}{2\sqrt{3}}\cos z+3k^4\sin z\right.\\&\left.-2k^4\sin 3z \bigg)\right]-\frac{\mu^2}{2}\left[-2\cos z+2ak^2\left(1+\cos 2z\right)+a^2k^4 \left(-2\cos z\right.\right.\\&\left.\left.+\cos 3z\right)\right]+O\left(\mu^3+a^3\right).
\end{aligned}
\end{equation}
Moreover,
\begin{equation*}
\begin{aligned}
T_{0,a} \sin z
&= \frac{2\lambda}{\sqrt{3}k} \partial_z \sin z- \partial_{zz} \sin z - \sin z + 2ak^2\left(\sin z \partial_z \sin z+ \partial_{zz} \left(\cos z\sin z\right)\right)+ a^2\\&\left(\frac{k^3\lambda}{2\sqrt{3}} \partial_z \sin z+ 2k^4 \sin 2z \partial_z \sin z - \frac{3k^4}{2}\partial_{zz} \sin z + k^4 \partial_{zz}\left(\cos 2z \sin z\right) \right)+O\left(a^3\right) \\
&=\frac{2\lambda}{\sqrt{3}k} \cos z + \sin z - \sin z + 2ak^2\left[\sin z\cos z +\left(-\cos z \sin z-2\sin z\cos z\right.\right.\\
&\left.\left.-\cos z\sin z\right)\right]+ a^2\bigg(\frac{k^3\lambda}{2\sqrt{3}} \cos z+2 k^4\sin 2z \cos z+ \frac{3k^4}{2}\sin z
+k^4\left(-4\cos 2z\right.\\&\left.\sin z- 4\sin 2z \cos z - \cos 2z \sin z\right) \bigg)+O\left(a^3\right)\\&
=\frac{2\lambda}{\sqrt{3}k}\cos z +2ak^2\left(-3\sin z\cos z\right) + a^2\bigg( \frac{k^3\lambda}{2\sqrt{3}} \cos z+ \frac{3k^4}{2}\sin z-2k^4 \cos z \sin 2z\\&+ k^4\left(-5\cos 2z \sin z\right)\bigg)+O\left(a^3\right)
\end{aligned}
\end{equation*}
\begin{equation*}
\begin{aligned}
\quad\quad\quad\quad&=\frac{2\lambda}{\sqrt{3}k} \cos z-3ak^2\sin 2z+ a^2\left(\frac{k^3\lambda}{2\sqrt{3}} \cos z+ \frac{3k^4}{2}\sin z -2k^4 \frac{\sin 3z+\sin z}{2}\right.\\&\left.-5k^4\frac{\sin 3z-\sin z}{2}\right)+O\left(a^3\right)\\&
=\frac{2\lambda}{\sqrt{3}k} \cos z -3ak^2\sin 2z+ a^2\left(\frac{k^3\lambda}{2\sqrt{3}} \cos z+ 3k^4\sin z -\frac{7}{2}k^4 \sin 3z\right)+O\left(a^3\right),
\end{aligned}
\end{equation*}
\begin{equation*}
\begin{aligned}
T_{1,a} \sin z
&=\frac{2\lambda}{\sqrt3 k}\sin z-2\p_z\sin z+2a k^2\left(\sin^2 z+2\partial_z\left(\cos z\sin z\right)\right)\notag +a^2\bigg(\frac{k^3\lambda}{2\sqrt3}\sin z\\&
+2k^4\sin 2z\sin z-3k^4\partial_z \sin z+2k^4\partial_z\left(\cos 2z\sin z\right)\bigg)+O\left(a^3\right)\\
&=\frac{2\lambda}{\sqrt{3}k}\sin z-2\cos z+2a k^2\left(\sin^2 z+2\left(-\sin^2 z+\cos^2 z\right)\right)+a^2\bigg(\frac{k^3\lambda}{2\sqrt{3}}\sin z\\&+2k^4\sin 2z\sin z-3k^4\cos z+2k^4\left(-2\sin 2z \sin z+\cos 2z \cos z\right)\bigg)+O\left(a^3\right)\\
&=\frac{2\lambda}{\sqrt{3}k}\sin z-2\cos z+2a k^2\left(-\sin^2 z+2\cos^2 z\right)+a^2\bigg(\frac{k^3\lambda}{2\sqrt{3}}\sin z-3k^4\cos z\\&-2k^4\sin 2z\sin z+2k^4\cos 2z \cos z\bigg)+O\left(a^3\right)\\&
=\frac{2\lambda}{\sqrt{3}k}\sin z-2\cos z+2a k^2\left(-\frac{1-\cos 2z}{2}+2\frac{1+\cos 2z}{2}\right)+a^2\left(\frac{k^3\lambda}{2\sqrt{3}}\sin z\right.\\&\left.-3k^4\cos z-2k^4\frac{\cos z-\cos 3z}{2}+2k^4\frac{\cos 3z+\cos z}{2}\right)+O\left(a^3\right)\\&
=\frac{2\lambda}{\sqrt{3}k}\sin z-2\cos z+a k^2\left(1+3\cos 2z\right)+a^2\bigg(\frac{k^3\lambda}{2\sqrt{3}}\sin z-3k^4\cos z\\&+2k^4\cos 3z\bigg)+O\left(a^3\right),
\end{aligned}
\end{equation*}
\begin{equation*}
\begin{aligned}
T_{2,a} \sin z
&=-2\sin z+4a k^2\cos z\sin z+a^2k^4\left(-3\sin z+2\cos 2z\sin z\right)+O\left(a^3\right)\quad\quad\quad\\&
=-2\sin z+2ak^2\sin 2z +a^2k^4\left(-3\sin z+2\frac{\sin 3z-\sin z}{2}\right)+O\left(a^3\right)\\&
=-2\sin z+2ak^2\sin 2z +a^2k^4\left(-4\sin z+\sin 3z\right)+O\left(a^3\right),
\end{aligned}
\end{equation*}
then we have
\begin{equation}\label{C.3}
\begin{aligned}
\mathcal{T}_{a, \mu}^{\lambda} \sin z&=T_{0, a}\sin z+i \mu T_{1, a}\sin z-\frac{\mu^2 }{2} T_{2, a}\sin z+O\left(\mu^3\right)\\&
=\frac{2\lambda}{\sqrt{3}k}\cos z- 3a k^2\sin 2z+ a^2\left( \frac{k^3\lambda}{2\sqrt{3}}\cos z+ 3k^4\sin z- \frac{7k^4}{2}\sin 3z \right)+\\
& i\mu\left[\frac{2\lambda}{\sqrt{3}k}\sin z- 2\cos z+a k^2\left(1+3\cos2z\right)+a^2\bigg(\frac{k^3\lambda}{2\sqrt{3}}\sin z- 3k^4\cos z\right.\\&\left.+ 2k^4\cos 3z \bigg)\right]-\frac{\mu^2}{2}\left[-2\sin z+2ak^2\sin 2z +a^2k^4\left(-4\sin z+\sin 3z\right)\right]\\&+O\left(\mu^3+a^3\right).
\end{aligned}
\end{equation}
Furthermore, from
\begin{equation*}
\begin{aligned}
T_{0,a} \cos 2z
&= \frac{2\lambda}{\sqrt{3}k} \partial_z \cos 2z - \partial_{zz} \cos 2z - \cos 2z + 2ak^2\left(\sin z \partial_z \cos 2z + \partial_{zz} \left(\cos z\cos 2z\right)\right)+\\& a^2\left(\frac{k^3\lambda}{2\sqrt{3}} \partial_z \cos 2z + 2k^4 \sin 2z \partial_z \cos 2z
- \frac{3k^4}{2}\partial_{zz} \cos 2z + k^4 \partial_{zz}(\cos^2 2z)\right) +O\left(a^3\right)\\
&= -\frac{2\lambda}{\sqrt{3}k}\cdot 2\sin 2z +4 \cos 2z - \cos 2z + 2ak^2\left(-2\sin z\sin 2z + \left(-\cos z\cos 2z\right.\right.\\&\left.\left.
+2\sin z \cdot 2\sin 2z-4\cos z\cos 2z\right)\right) + a^2\bigg(-\frac{k^3\lambda}{\sqrt{3}} \sin 2z - 4k^4 \sin^2 2z+ 6k^4\cos 2z\\&+ 8k^4\left(\sin^2 2z - \cos^2 2z\right)\bigg)+O\left(a^3\right)\\
&=-\frac{4\lambda}{\sqrt{3}k} \sin 2z +3\cos 2z + 2ak^2\left(2\sin z\sin 2z -5\cos z\cos 2z\right)+ a^2\\&\left(-\frac{k^3\lambda}{\sqrt{3}} \sin 2z- 4k^4 \sin^2 2z+ 6k^4\cos 2z+ 8k^4\left(\sin^2 2z - \cos^2 2z\right)\right)+O\left(a^3\right)\\&
=-\frac{4\lambda}{\sqrt{3}k} \sin 2z +3\cos 2z + 2ak^2\left(2\frac{\cos z-\cos 3z}{2} -5\frac{\cos z+\cos 3z}{2}\right)\\&+ a^2\left(-\frac{k^3\lambda}{\sqrt{3}} \sin 2z+4k^4 \sin^2 2z+ 6k^4\cos 2z- 8k^4 \cos^2 2z\right)+O\left(a^3\right)\\&
=-\frac{4\lambda}{\sqrt{3}k} \sin 2z +3\cos 2z - ak^2\left(3\cos z+7\cos 3z\right) + a^2\bigg(-\frac{k^3\lambda}{\sqrt{3}}\sin 2z\\&+ 6k^4\cos 2z-2k^4- 6k^4 \cos 4z\bigg)+O\left(a^3\right),
\end{aligned}
\end{equation*}
\begin{equation*}
\begin{aligned}
T_{1,a} \cos 2z
&=\frac{2\lambda}{\sqrt3 k}\cos 2z-2\p_z\cos 2z+2a k^2\left(\sin z\cos 2z+2\partial_z(\cos z\cos 2z)\right)+a^2\\&\left(\frac{k^3\lambda}{2\sqrt3}\cos 2z
+2k^4\sin 2z\cos 2z-3k^4\partial_z \cos 2z+2k^4\partial_z\cos^2 2z\right)+O\left(a^3\right)\\&
=\frac{2\lambda}{\sqrt{3}k}\cos 2z+4\sin 2z+2a k^2\left(\sin z\cos 2z+2\left(-\sin z\cos 2z-2\cos z\sin 2z\right)\right)+\\&a^2\left(\frac{k^3\lambda}{2\sqrt{3}}\cos 2z+2k^4\sin 2z\cos 2z+6k^4\sin 2z+2k^4(2\cos 2z)(-2\sin 2z)\right)+O\left(a^3\right)\\&
=\frac{2\lambda}{\sqrt{3}k}\cos 2z+4\sin 2z+2a k^2\left(-\sin z\cos 2z-4\cos z\sin 2z\right)+a^2\\&\left(\frac{k^3\lambda}{2\sqrt{3}}\cos 2z+6k^4\sin 2z+2k^4\sin 2z\cos 2z-8k^4\cos 2z\sin 2z\right)+O\left(a^3\right)\\
&=\frac{2\lambda}{\sqrt{3}k}\cos 2z+4\sin 2z+2a k^2\left(-\frac{\sin 3z-\sin z}{2}-4\frac{\sin 3z+\sin z}{2}\right) +a^2\\&\left(\frac{k^3\lambda}{2\sqrt{3}}\cos 2z+6k^4\sin 2z-6k^4\cos 2z\sin 2z\right)+O\left(a^3\right)\\&
=\frac{2\lambda}{\sqrt{3}k}\cos 2z+4\sin 2z+a k^2\left(-3\sin z-5\sin 3z\right)+a^2\bigg(\frac{k^3\lambda}{2\sqrt{3}}\cos 2z
\end{aligned}
\end{equation*}
\begin{equation*}
\begin{aligned}
&+6k^4\sin 2z-3k^4\sin 4z\bigg)+O\left(a^3\right)\\
&=\frac{2\lambda}{\sqrt{3}k}\cos 2z+4\sin 2z-a k^2\left(3\sin z+5\sin 3z\right) +a^2\bigg(\frac{k^3\lambda}{2\sqrt{3}}\cos 2z\\&+6k^4\sin 2z-3k^4\sin 4z\bigg)+O\left(a^3\right),
\end{aligned}
\end{equation*}
\begin{equation*}
\begin{aligned}
T_{2,a} \cos 2z
&=-2\cos 2z+4a k^2\cos z\cos 2z+a^2k^4\left(-3\cos 2z+2\cos^2 2z\right)+O\left(a^3\right)\\&
=-2\cos 2z+4ak^2\frac{\cos 3z+\cos z}{2} +a^2k^4\left(-3\cos 2z+2\frac{1+\cos 4z}{2}\right)+O\left(a^3\right)\\&
=-2\cos 2z+2ak^2\left(\cos z+\cos 3z\right) +a^2k^4\left(1-3\cos 2z+\cos 4z\right)+O\left(a^3\right),
\end{aligned}
\end{equation*}
we have
\begin{equation}\label{C.4}
\begin{aligned}
\mathcal{T}_{a, \mu}^{\lambda} \cos 2z&=T_{0, a}\cos 2z+i \mu T_{1, a}\cos 2z-\frac{\mu^2}{2} T_{2, a}\cos 2z+O\left(\mu^3\right)\\&
=-\frac{4\lambda}{\sqrt{3}k}\sin 2z+3\cos2z-a k^2\left(3\cos z+7\cos 3z\right)+a^2\bigg( -\frac{k^3\lambda}{\sqrt{3}}\sin 2z+\\&6k^4\cos 2z-2k^4- 6k^4\cos 4z\bigg) +i\mu\left[\frac{2\lambda}{\sqrt{3}k}\cos 2z+4\sin 2z-a k^2\left(3\sin z\right.\right.\\&\left.\left.+5\sin 3z\right)+a^2\left(\frac{k^3\lambda}{2\sqrt{3}}\cos 2z+6k^4\sin 2z-3k^4\sin 4z \right)\right]-\frac{\mu^2}{2}\left[-2\cos 2z\right.\\&\left.+2ak^2\left(\cos z+\cos 3z\right)+a^2k^4\left(1-3\cos 2z+\cos 4z\right)\right]+O\left(\mu^3+a^3\right).
\end{aligned}
\end{equation}
and from
\begin{equation*}
\begin{aligned}
T_{0,a} \sin 2z
&= \frac{2\lambda}{\sqrt{3}k} \partial_z \sin 2z- \partial_{zz} \sin 2z - \sin 2z + 2ak^2\left(\sin z \partial_z \sin 2z+ \partial_{zz}\left(\cos z\sin 2z\right)\right)\\
&+a^2\left(\frac{k^3\lambda}{2\sqrt{3}} \partial_z \sin 2z+ 2k^4 \sin 2z \partial_z \sin 2z-\frac{3k^4}{2}\partial_{zz} \sin 2z + k^4 \partial_{zz}\left(\cos 2z \sin 2z\right)\right)\\&+O\left(a^3\right) \\&
=\frac{4\lambda}{\sqrt{3}k} \cos 2z+4\sin 2z - \sin 2z + 2ak^2\left(2\sin z \cos 2z+\left(-\cos z\sin 2z-4\sin z \right.\right.\\&\left.\left.\cos 2z-4\cos z \sin 2z\right)\right)+ a^2\bigg(\frac{k^3\lambda}{\sqrt{3}} \cos 2z+ 4k^4 \sin 2z \cos 2z+ 6k^4\sin 2z+ k^4\\&\left(-4\cos 2z \sin 2z-8\sin 2z\cos 2z-4\cos 2z\sin 2z\right) \bigg)+O\left(a^3\right)\\&
=\frac{4\lambda}{\sqrt{3}k} \cos 2z+3\sin 2z + 2ak^2\bigg(-2\sin z \cos 2z-5\cos z\sin 2z\bigg)\\
&+a^2\left(\frac{k^3\lambda}{\sqrt{3}}\cos 2z+ 6k^4\sin 2z-12k^4 \sin 2z \cos 2z \right)+O\left(a^3\right) \\&
=\frac{4\lambda}{\sqrt{3}k} \cos 2z+3\sin 2z - ak^2\left(3\sin z +7\sin 3z\right)+ a^2\bigg(\frac{k^3\lambda}{\sqrt{3}}\cos 2z\\&+ 6k^4\sin 2z-6k^4\sin 4z\bigg)+O\left(a^3\right),
\end{aligned}
\end{equation*}
\begin{equation*}
\begin{aligned}
T_{1,a} \sin 2z
&=\frac{2\lambda}{\sqrt3 k}\sin 2z-2\p_z\sin 2z+2a k^2\left(\sin z\sin 2z+2\partial_z\left(\cos z\sin 2z\right)\right)+a^2\\&\left(\frac{k^3\lambda}{2\sqrt3}\sin 2z+2k^4\sin^2 2z-3k^4\partial_z \sin 2z+2k^4\partial_z\left(\cos 2z\sin 2z\right)\right)+O\left(a^3\right)\\&
=\frac{2\lambda}{\sqrt{3}k}\sin 2z-4\cos 2z+2a k^2\left(\sin z\sin 2z+2\left(-\sin z\sin 2z+2\cos z\cos 2z\right)\right)\\&+a^2\left(\frac{k^3\lambda}{2\sqrt{3}}\sin 2z+2k^4\sin^2 2z-6k^4\cos 2z+2k^4\left(-2\sin^2 2z+2\cos^2 2z\right)\right)\\&+O\left(a^3\right)\\
&=\frac{2\lambda}{\sqrt{3}k}\sin 2z-4\cos 2z+2a k^2\left(-\sin z\sin 2z+4\cos z\cos 2z\right)+a^2\left(\frac{k^3\lambda}{2\sqrt{3}}\sin 2z\right.\\&\left.+2k^4\frac{1-\cos 4z}{2}-6k^4\cos 2z+2k^4\left(-2\frac{1-\cos 4z}{2}+2\frac{1+\cos 4z}{2}\right)\right)+O\left(a^3\right)\\&
=\frac{2\lambda}{\sqrt{3}k}\sin 2z-4\cos 2z+2a k^2\left(-\frac{\cos z-\cos 3z}{2}+4\frac{\cos z+\cos 3z}{2}\right)+a^2\\&\left(\frac{k^3\lambda}{2\sqrt{3}}\sin 2z+k^4+3k^4\cos 4z-6k^4\cos 2z\right)+O\left(a^3\right)\\
&=\frac{2\lambda}{\sqrt{3}k}\sin 2z-4\cos 2z+a k^2\left(3\cos z+5\cos 3z\right)+a^2\bigg(\frac{k^3\lambda}{2\sqrt{3}}\sin 2z-6k^4\cos 2z\\&+3k^4\cos 4z+k^4\bigg)+O\left(a^3\right),
\end{aligned}
\end{equation*}
\begin{equation*}
\begin{aligned}
T_{2,a} \sin 2z
&=-2\sin 2z+4a k^2\cos z\sin 2z+a^2k^4\left(-3\sin 2z+2\cos 2z\sin 2z\right)+O\left(a^3\right)\quad\quad\quad\\&
=-2\sin 2z+4ak^2\frac{\sin 3z+\sin z}{2} +a^2k^4\left(-3\sin 2z+\sin 4z\right)+O\left(a^3\right)\\&
=-2\sin 2z+2ak^2\left(\sin z+\sin 3z\right) +a^2k^4\left(-3\sin 2z+\sin 4z\right)+O\left(a^3\right),\quad\quad\quad\quad\quad\quad
\end{aligned}
\end{equation*}
we obtain
\begin{equation}\label{C.5}
\begin{aligned}
\mathcal{T}_{a, \mu}^{\lambda} \sin 2z&=T_{0, a}\sin 2z+i \mu T_{1, a}\sin 2z-\frac{\mu^2 }{2} T_{2, a}\sin 2z+O\left(\mu^3\right)\\
&=\frac{4\lambda}{\sqrt{3}k}\cos 2z + 3\sin 2z - a k^2\left(3\sin z + 7\sin 3z\right) + a^2\bigg(\frac{k^3\lambda}{\sqrt{3}}\cos 2z+ 6k^4\\&\sin 2z - 6k^4\sin 4z\bigg)+ i\mu\left[\frac{2\lambda}{\sqrt{3}k}\sin 2z - 4\cos 2z + a k^2(3\cos z+5\cos 3z)\right.\\&\left.+ a^2\left(\frac{k^3\lambda}{2\sqrt{3}}\sin 2z - 6k^4\cos 2z + 3k^4\cos 4z + k^4 \right)\right]-\frac{\mu^2}{2}\left[-2\sin 2z \right.\\&\left.+ 2ak^2\left(\sin z +\sin 3z\right) + a^2k^4\left(-3\sin 2z + \sin 4z\right)\right]+O\left(\mu^3+a^3\right).
\end{aligned}
\end{equation}

\vspace{0.5cm}
\noindent {\bf Acknowledgements}
LLF was partially supported by NSFC Grant No. 12426607 and the NSF of Henan Province of China (Grant No. 252300421218). HJG was partially supported by the Jiangsu Provincial Scientific Research Center of Applied Mathematics under Grant No. BK20233002.

\vspace{0.5cm}
\noindent {\bf Conflict of interest}
The authors declare that there is no conflict of interest.

\vspace{0.5cm}
\noindent {\bf Data Availability}
There is no data associated to this work.

\bibliographystyle{siam}
\bibliography{bKP}


\end{document}